\documentclass[12pt, reqno]{amsart}
\usepackage{amsmath, amsthm, amscd, amsfonts, amssymb, graphicx, color, mathrsfs}
\usepackage[bookmarksnumbered, colorlinks, plainpages]{hyperref}
\usepackage[all]{xy} 
\usepackage{slashed}

\newtheorem{theorem}{Theorem}[section]
\newtheorem{lemma}[theorem]{Lemma}

\newtheorem{proposition}[theorem]{Proposition}
\newtheorem{corollary}[theorem]{Corollary}

\theoremstyle{definition}

\theoremstyle{remark}
\newtheorem{remark}[theorem]{Remark}
\numberwithin{equation}{section}

\providecommand{\keywords}[1]
{
  \small	
  \textbf{\textit{Keywords---}} #1
}

\begin{document}
\setcounter{page}{1}

\title[Scaling limits for 1d random Schrödinger operators]{More scaling limits for 1d random Schrödinger operators with critically decaying and vanishing potential}

\author[Y.H]{Yi Han}
\address{
  Yi Han:
  \endgraf
  Department of Mathematics, Massachusetts Institute of Technology, Cambridge MA
  \endgraf
  {\it E-mail address} {\rm hanyi16@mit.edu}
  }

\thanks{The research was conducted when the author was supported in part by EPSRC grant EP/W524141/1 and in part by Simons Foundation grant (601948, DJ)
}

\begin{abstract} Consider the random Schrödinger operator $H_n$ defined on $\{0,1,\cdots,n\}\subset\mathbb{Z}$
$$
    (H_n\psi)_\ell=\psi_{\ell-1,n}+\psi_{\ell+1,n}+\sigma\frac{\omega_\ell}{a_{\ell,n}}\psi_{\ell,n},\quad \psi_0=\psi_{n+1}=0,
$$
where $\sigma>0$, $\omega_\ell$ are i.i.d. random variables and $a_{\ell,n}$ typically has order $\sqrt{n}$ for $\ell\in[\epsilon n,(1-\epsilon)n]$ and any $\epsilon>0$. Two important cases: (a) the vanishing case $a_{\ell,n}=\sqrt{n}$ and (b) the decaying case $a_{\ell,n}=\sqrt{\ell}$, were studied before in \cite{kritchevski2011scaling}. In this paper we consider more general decaying profiles that lie in between these two extreme cases. We characterize the scaling limit of transfer matrices and determine the point process limit of eigenvalues near a fixed energy in the bulk, in terms of solutions to coupled SDEs. We obtain new point processes that share similar properties to the $\text{Sch}_\tau$ process. We determine the shape profile of eigenfunctions after a suitable rescaling, that corresponds to a uniformly chosen eigenvalue of  $H_n$. We also give a more detailed description of the newly defined point processes, including the probability of small and large gaps and a variance estimate.
\end{abstract} 

\maketitle

\allowdisplaybreaks



\section{Introduction}
To study the transition from the localized to the delocalized phase of Anderson operators, Kritchevski, Valkó and Virag \cite{kritchevski2011scaling} introduced the following random Schrödinger operator defined on $\mathbb{Z}_n:=\{0,1,\cdots,n\}\subset\mathbb{Z}$
\begin{equation}
    (H_n\psi)_\ell=\psi_{\ell-1,n}+\psi_{\ell+1,n}+v_{\ell,n}\psi_{\ell,n},\quad \psi_0=\psi_{n+1}=0,
\end{equation}
and considered the \textbf{vanishing model} with potential 
\begin{equation}\label{vanishing}
    v_{\ell,n}=\sigma \omega_\ell/n^\alpha,
\end{equation}
and the \textbf{decaying model} with potential
\begin{equation}\label{decaying}
    \quad v_{\ell,n}=\sigma \omega_\ell/\ell^\alpha
\end{equation}
for some $\alpha>0$. The random variables $\omega_k$ are independent, mean 0 and variance 1 random variables. The phase transition takes place at the value $\alpha=\frac{1}{2}$, where interesting point process scaling limits of the eigenvalue distributions are expected. Prior to the work \cite{kritchevski2011scaling}, the spectral properties of the decaying model were investigated in \cite{kotani1988one}, \cite{kiselev1998modified}, showing that for the infinite volume version of the decaying model, the spectrum is absolutely continuous for $\alpha>\frac{1}{2},$ is pure point for $\alpha\in(0,\frac{1}{2})$ and has mixed behavior for $\alpha=\frac{1}{2}.$ Moreover, Poisson statistics is expected for eigenvalues of the decaying model when $\alpha\in(0,\frac{1}{2})$, in the sense of \cite{minami1996local}, \cite{germinet2014spectral}. For the classical decaying model, one should also mention the works \cite{killip2007eigenfunction} and \cite{minami2007theory}.

As the potentials $v_{k,n}$ are vanishing, the eigenvalues of $H_n$ should approximate that of the free Laplacian, which has density of states described by the arcsine law $\frac{\rho}{2\pi}$ where
\begin{equation}\label{largevolume}
    \rho=\rho(E)=1/\sqrt{1-E^2/4}.
\end{equation}

To obtain nontrivial scaling limits, we consider the spectrum $\Lambda_n$ of the rescaled operator
\begin{equation}
    \rho n(H_n-E)
\end{equation} for any $E\in(-2,2)$.
Define
$$z=E/2+i\sqrt{1-(E/2)^2},$$
then it is proved in \cite{kritchevski2011scaling} that for the vanishing model \eqref{vanishing} with $\alpha=\frac{1}{2}$, for $0<|E|<2$ and with $\tau=(\sigma\rho)^2$, we have the convergence in distribution of the point process $\Lambda_n-\arg(z^{2n+2})-\pi$ to the point process
$\text{Sch}_\tau:=\{\lambda:\varphi^{\lambda/\tau}(\tau)\in 2\pi\mathbb{Z}\}$,
where $\tau=(\sigma\rho)^2$ and $\varphi^{\lambda}$ solves the family of SDEs
\begin{equation}
    d\varphi^{\lambda}(t)=\lambda dt+d\mathcal{B}+\text{Re}[e^{-i\varphi^\lambda(t)}d\mathcal{W}],\quad\varphi^\lambda(0)=0
\end{equation}
and $\mathcal{B}$,$\mathcal{W}$ are independent real and complex standard Brownian motions. For the decaying model \eqref{decaying} with $\alpha=\frac{1}{2}$, they proved the point process $\Lambda_n$ converges to a point process $\Lambda$ that agrees with the $\text{Sine}_\beta$ process from random matrix theory \cite{valko2009continuum}, with $\beta=\frac{2}{(\sigma\rho)^2}$.

 When $\alpha>\frac{1}{2}$ and the potentials decay faster, $\Lambda_n$ converges to a deterministic limit, called the clock (or picket fence) process
$$\text{clock}_\eta=\{2\pi k+\pi +2\eta,k\in\mathbb{Z}\}.$$ See \cite{flores2023one}, \cite{killip2009eigenvalue}, \cite{last2008fine} for convergence to the Clock process in related discrete models, and see \cite{nakano2014level} and \cite{MR3379349} for convergence to the Clock process and Sch (resp. $\operatorname{Sine}_\beta$) processes on vanishing (resp. decaying) continuum models in supercritical and critical cases.
When $\alpha\in(0,\frac{1}{2})$, then the potentials decay slower than the critical value, and the point process $\Lambda_n$ should have a Poisson limit, though the Poisson limit is currently only proved in \cite{kotani2017poisson} for a continuous time model with randomness arising from a Brownian motion. There is also a recent study on the scaling limit of eigenvalues at the edge $\{\pm 2\}$ where a different scaling will arise, see \cite{han2023universal}.

We will now further investigate the intermediate case $\alpha=\frac{1}{2}$. Inspired by the vanishing and decaying models, we consider the \textbf{mixed vanishing-decaying model}  
\begin{equation}
    \label{mixedvanishingdecaying}
v_{k,n}=\sigma\frac{\omega_k}{n^\eta(n+1-k)^{\frac{1}{2}-\eta}}\end{equation}for some $\eta\in[0,\frac{1}{2}]$. In choosing $\eta=\frac{1}{2}$ we recover the vanishing model, and in choosing $\eta=0$ and relabeling $\ell$ by $n+1-\ell$, we recover the decaying model. 

To simplify some computations, in this paper we use the convention that, at site $\ell$ the weight is $\frac{1}{n^\eta(n+1-\ell)^{\frac{1}{2}-\eta}}$ rather than the more natural choice $\frac{1}{n^\eta \ell^{\frac{1}{2}-\eta}}$. A great portion of the computations in this paper can be carried over to more general decaying profiles that are locally slow varying and have order $1/{\sqrt{n}}$ in most of the sites.

Throughout this paper we assume the random variables $\omega_k$ are independent, have mean $0$ and variance $1$, and their third moments are bounded.

\subsection{Scaling limits of transfer matrices}
The first step to study our mixed vanishing-decaying model $H_n$ is to work out the scaling limits of transfer matrices. For $x\in\mathbb{R}$ define
$$T(x):=\begin{pmatrix}x&-1\\1&0\end{pmatrix}.$$

We consider eigenvalues $E+\frac{\lambda}{\rho n}$ near some bulk energy $E$, and define 
\begin{equation}\label{variableis}
\epsilon_\ell=\frac{\lambda}{\rho n}-\frac{\sigma\omega_\ell}{n^\eta (n+1-\ell)^{\frac{1}{2}-\eta}}.\end{equation}
Then the transfer matrix is 
\begin{equation}\label{transferis}
M_\ell^\lambda=T(E+\epsilon_\ell)T(E+\epsilon_{\ell-1})\cdots T(E+\epsilon_1),\quad 0\leq \ell\leq n,
\end{equation}
and $E+\frac{\lambda}{\rho n}$ is an eigenvalue of $H_n$ if and only if 
$$M_n^\lambda\begin{pmatrix} 1\\0\end{pmatrix}\parallel \begin{pmatrix}0\\1\end{pmatrix}.$$
To determine a limit of $M_\ell^\lambda$ as $n\to\infty$, we need to work with $Q_\ell^\lambda$ defined as
\begin{equation}\label{307rescaling}
Q_\ell^\lambda=T^{-\ell}(E)M_\ell^\lambda,\quad 0\leq\ell\leq n.\end{equation}
and study the scaling limit of $Q_\ell^\lambda$.

For any $E\in(-2,2)$, it will be convenient to diagonalize $T(E):=ZDZ^{-1}$, where
\begin{equation} \label{Ediagonalized}
    D=\begin{pmatrix}
        \bar{z}&0\\0&z
    \end{pmatrix},\quad Z=\frac{i\rho(E)}{2}\begin{pmatrix}\bar{z}&-z\\1&-1\end{pmatrix},\quad Z^{-1}=\begin{pmatrix}1&-z\\1&-\bar{z}\end{pmatrix},\quad z=\frac{E}{2}+i\sqrt{1-\frac{E^2}{4}}.
\end{equation}

We prove:

\begin{theorem}\label{convergence2.1}
First assume $0<|E|<2$ and $\eta\in(0,\frac{1}{2}]$. Given $\mathcal{B}(t),$ $\mathcal{B}_2(t)$ and $\mathcal{B}_3(t)$ standard Brownian motions in $\mathbb{R}$ and mutually independent, denote by $\mathcal{W}(t)=\frac{1}{\sqrt{2}}(\mathcal{B}_2(t)+i\mathcal{B}_3(t)).$ Then 
$$(Q^\lambda_{\lfloor nt\rfloor},0\leq t\leq 1)\Rightarrow (Q^\lambda(t),0\leq t\leq 1),\quad n\to\infty,\quad $$
where $Q^\lambda$ is the unique strong solution to the SDE
\begin{equation}\label{noisecoefficientofQ}
    dQ^\lambda=\frac{1}{2}Z\left(\begin{pmatrix} i\lambda&0\\0&-i\lambda\end{pmatrix} dt+\frac{\sigma\rho}{(1-t)^{\frac{1}{2}-\eta}}\begin{pmatrix}
        id\mathcal{B}& d\mathcal{W}\\d\bar{\mathcal{W}}&-id\mathcal{B} \end{pmatrix}\right)Z^{-1}Q^\lambda,\quad Q^\lambda(0)=I.
\end{equation}
Note that although the coefficient of the SDE solved by $Q^\lambda$ grows to infinity at $t=1$, the quadratic variation of the process is still finite as the integral $\int_0^1\frac{1}{(1-t)^{1-2\eta}}<\infty$ and $\|Z\|$ is bounded, so the value $Q^\lambda(1)$ can be uniquely defined as the almost sure limit of $\lim_{t\uparrow 1}Q^\lambda(t)$.

The stated convergence takes place in $\lambda$- finite dimensional distributions, uniformly in $t\in[0,1-\epsilon]$ for any $\epsilon>0$. 

In the case $E=0$, we still have the convergence but now $Q^\lambda$ solves the SDE
\begin{equation}
\label{equationQlambda}dQ^\lambda=\frac{1}{2}Z\left(\begin{pmatrix} i\lambda&0\\0&-i\lambda\end{pmatrix} dt+\frac{\sigma\rho}{(1-t)^{\frac{1}{2}-\eta}}\begin{pmatrix}
        id\mathcal{B}_1& id\mathcal{B}_2\\-id\mathcal{B}_2&-id\mathcal{B}_1 \end{pmatrix}\right)Z^{-1}Q^\lambda,\quad Q^\lambda(0)=I.
\end{equation}
where $\mathcal{B}_1,\mathcal{B}_2$ are two independent real Brownian motions.
\end{theorem}

In the meantime, we will establish tightness estimates for the transfer matrices and expected number of eigenvalues in a small interval. We prove:

\begin{theorem}\label{theorem3.1good}
Assume $\eta\in(0,\frac{1}{2}]$ and $\sigma<\infty$. Then 
    \begin{enumerate}
        \item  Fix $E\in(-2,2)$ and $R<\infty$, we can find $c>0$ such that for any $t>0$ and $n$ large, on an event of probability no less than $1-c/t$, 
        \begin{enumerate}
            \item there holds
      \begin{equation}
            \max_{0\leq\ell\leq n,|\lambda|\leq R}\text{Tr}M_\ell^\lambda M_\ell^{\lambda*}<t.
        \end{equation}
        \item Given an eigenvector $\psi$ of $H_n$ that is $L^2$-normalized and has eigenvalue in $[E-\frac{R}{\rho n},E+\frac{R}{\rho n}]$, 
            \begin{equation}\label{eigenpairups}
                \frac{2}{(n+1)t^2}<|\psi_\ell|^2+|\psi_{\ell+1}|^2<\frac{2t^2}{n+1},\quad 0\leq\ell\leq n.
            \end{equation}
        \end{enumerate}
        \item  For any $R>0$ consider the interval $\Delta_n(E)=(E-\frac{R}{n\rho(E)},E+\frac{R}{n\rho(E)})$. Denote by $N_n(E)$ the number of eigenvalues of $H_n$ belonging to $\Delta_n(E)$. Then for arbitrary $\epsilon>0$,
        \begin{equation}\label{eigenvectorupper}
        \sup_n\sup_{E\in(-2+\epsilon,2-\epsilon)}\mathbb{E}[N_n(E)^{3/2}]<\infty.\end{equation}
    \end{enumerate}
\end{theorem}
The estimate \ref{eigenvectorupper} leads to Wegner's estimate (see for example \cite{kirsch2007invitation}) for the Anderson operator $H_n$. These estimates are essential for us to establish process level scaling limits, as stated in the next paragraph.

\subsection{The point process limit}
Now we investigate the point process limit of $\Lambda_n$ as $n\to\infty$. The limiting point process is not the Clock process or Poisson process, and does not match existing point processes from random matrix models. Rather, they are characterized as zeros of complex analytic functions.

We introduce the following set of Prüfer coordinates: Since for any real vector $(a,b)^t$,
$$Z^{-1}\begin{pmatrix} a\\ b\end{pmatrix}=\begin{pmatrix} a-bz\\ \overline{a-bz}\end{pmatrix},$$
we shall write
\begin{equation}\label{nicestory}
    \begin{pmatrix}
q^\lambda(t)\\\overline{q^\lambda(t)}\end{pmatrix}:=Z^{-1}Q(\lambda,t)\begin{pmatrix}1\\0
    \end{pmatrix}
\end{equation}
where $q^\lambda(t)$ satisfies $q^\lambda(0)=1$. Define also 
\begin{equation}\label{rlambdathetalambda}(q^\lambda(t))^2=e^{r^\lambda(t)+i\theta^\lambda(t)}.\end{equation}

It will be shown in Proposition \ref{whatisthesde?} that $\theta^\lambda$ solves the SDE
\begin{equation}\label{dphilambdat}
    d\theta^\lambda(t)=\lambda dt+\frac{\sigma\rho}{(1-t)^{\frac{1}{2}-\eta}}[d\mathcal{B}+\text{Im}[e^{-i\theta^\lambda(t)}d\mathcal{W}]],\quad \theta^\lambda(0)=0.
\end{equation}

We note that by a change of variable $t=1-f(s)$, $f(s):=(1-2\eta s)^\frac{1}{2\eta}$, then the relative phase function $\alpha_s^\lambda:=\theta^\lambda(s)-\theta^0(s)$ solves the SDE
$$
d\alpha_s^\lambda=\lambda(-f'(s))ds+\theta\rho\text{Im}[(e^{-i\alpha_s^\lambda}-1)d\widetilde{W}_s],\quad d\widetilde{W}_s:=e^{-i\theta^0(s)}dW_s,
$$ and the latter has the form of the Brownian Carousel as in equation (2) of \cite{valko2009continuum}. Thus the point process $\eta Sch^*$ can be described by a Brownian carousel.

 Now we define the following point process analogous to the $sch_\tau$ process in \cite{kritchevski2011scaling}: for each $\phi\in[0,2\pi),$ define
\begin{equation}\label{etaschphi}
    \eta Sch^\phi:=\{\lambda\in\mathbb{R}:\theta^\lambda(1)\in 2\pi\mathbb{Z}+\phi\}, 
\end{equation}
and simply denote $\eta Sch$ for $\eta Sch^0$.

We will also consider $$\eta Sch^*=^d\eta Sch^U$$ where $U$ is uniform random variable on $[0,2\pi]$ independent of all other random variables.

Note that by a slight abuse of notation, the definition of $\eta Sch$ also involves $\sigma$ and $\rho$. We choose to keep the dependence throughout this article as it will not cause confusions and will clarify what coefficients the estimates depend on.

The following corollary shows that the point process $\Lambda_{n,E}$ converges, up to a subsequence of $n$, to a counting measure defined with respect to the SDE $Q^\lambda$.

\begin{corollary}\label{convergenceprocess}
    For any subsequence $n_j$ such that $z^{n_j}$ converges, we have $T^{n_j}(E)\to\tilde{T}$ for some $2\times 2$ matrix $\tilde{T}$. Then the following claims are true:\begin{enumerate}
        \item The analytic function $M_{n_j}^\lambda$ (with $n=n_j$ in the definition \eqref{variableis},\eqref{transferis}) has a distributional limit $\tilde{T}Q^\lambda(1)$. \item The point process $\Lambda_{n_j}$ converges in distribution to the counting measure on the zero set of the analytic function $\lambda\mapsto [\tilde{T}Q^\lambda(1)]_{11}.$
        \item Define an analytic function $Q_{n,E}(\lambda,t)$ as the piecewise linear interpolation of the step function $Q^\lambda_{\lfloor nt\rfloor}$ defined in \eqref{307rescaling}. Then $Q_{n,E}(\lambda,t)$ converges to $Q(\lambda,t)$ in distribution for the uniform topology of $\mathbb{C}\times[0,1-\epsilon]$ and any $\epsilon>0$.
    \end{enumerate}
\end{corollary}

This technical Corollary \ref{convergenceprocess} sets up the ground for proving the following distributional limit of the rescaled point process $\Lambda_n$:

\begin{corollary}\label{corollary1.4}
    Recall that $\Lambda_n$ is the spectrum of the rescaled operator $\rho n(H_n-E)$. Then, whenever $0<|E|<2$ and under the assumption that $\omega_k$ are independent mean 0 variance 1 random variables with bounded third moment, the point process $\Lambda_n-\arg(z^{2n+2})-\pi$ converges in distribution to the point process 
    $\eta Sch$ defined in \eqref{etaschphi}.
\end{corollary}

\subsection{Shape of eigenfunctions for the vanishing-decaying model}

Having worked out the point process limit of eigenvalues of $H_n$ near a bulk energy $E$, we also study the joint scaling limit of eigenvalue-eigenvector pairs. This has been done in Rifkind and Virag \cite{rifkind2018eigenvectors} for the vanishing model in discrete time, and recently in Nakano \cite{nakano2022shape} in a continuous time model driven by Brownian motion, for both vanishing and decaying model. In this paper we generalize these results to the mixed vanishing-decaying model \eqref{mixedvanishingdecaying}.

For $\mu$ some eigenvalue of $H_n$ and the corresponding eigenvector $\psi^\mu$, we consider a probability measure on $[0,1]$ with density 
$$n|\psi^\mu(\lfloor nt\rfloor)|^2dt$$
where the leading factor $n$ is added to normalize it into a probability measure. Consider $\mathcal{M}([0,1])$ the space of probability measures on $[0,1]$, endowed with the weak topology, that is, $\mu_n\to\mu$ if and only if $\int fd\mu_n\to fd\mu$ for all $f\in\mathcal{C}_b^\infty([0,1];\mathbb{R})$.
The main result concerns joint convergence in distribution of the eigenvalue-eigenvector pair
$$(\mu,n|\psi^\mu(\lfloor nt\rfloor)^2|dt)\in\mathbb{R}\times\mathcal{M}([0,1])$$
once we uniformly choose $\mu$ out of eigenvalues of $H_n$.

\begin{theorem}\label{shapetheorem}
    Given a random variable $E$ following the arcsine law and $U$ the uniform distribution on $[0,1]$. Consider $\mathcal{Z}$ a standard two-sided Brownian motion. Assume $E,\mathcal{Z},U$ are mutually independent. 

    Uniformly choose $\mu$ from eigenvalues of $H_n$ and denote by $\psi^\mu$ the normalized eigenvector. Then we have the convergence in distribution
    $$\begin{aligned} &\left(\mu,n|\psi^\mu(\lfloor nt\rfloor)|^2 dt\right)\\
    &\Rightarrow \left(E,
\frac{\exp\left(\mathcal{Z}_{\frac{\sigma^2\rho^2}{4\eta}[(1-U)^{2\eta}-(1-t)^{2\eta}]}-\frac{\sigma^2\rho^2}{8\eta}\left|(1-U)^{2\eta}-(1-t)^{2\eta}\right|\right)dt}{\int_0^1 ds \exp\left(\mathcal{Z}_{\frac{\sigma^2\rho^2}{4\eta}[(1-U)^{2\eta}-(1-s)^{2\eta}]}-\frac{\sigma^2\rho^2}{8\eta}\left|(1-U)^{2\eta}-(1-s)^{2\eta}\right|\right)}
\right)
 \end{aligned}   $$
\end{theorem}

\begin{remark} If we take $\eta=\frac{1}{2}$ (i.e., the vanishing model), and $\sigma\rho=1$, we obtain after Brownian scaling the expression $\frac{1}{\sqrt{2}}\mathcal{Z}_{t-U}-\frac{1}{4}|t-U|$ in the exponential. This is consistent with the result of \cite{rifkind2018eigenvectors}. If we take $\eta\to 0$, we get $\frac{\log|1-U|}{\log|1-t|}$ terms in the exponential, which is consistent with the result of \cite{nakano2022shape}, Theorem 1.1 after a time reversal (the coefficients are different because \cite{nakano2022shape} considered a continuous time model).
\end{remark}

We may also consider potentials that vanish faster than the critical exponent $
    \frac{1}{2}$.
\begin{corollary}\label{reallyfast}
    Now we consider, instead of \eqref{mixedvanishingdecaying},
    \begin{equation}
        v_{k,n}:=\frac{\sigma\omega_k}{n^\eta (n+1-k)^{\tau-\eta}}
    \end{equation}
    for some given $\tau>\frac{1}{2}$ and $\eta\in[0,\tau]$. Then the point process $\Lambda_{n,E}$ converges in distribution to a Clock process, that is, $\Lambda_{n,E}-z^{2n+2}$ converges in distribution to $2\pi\mathbb{Z}$. Theorem \ref{convergence2.1} holds straightforwardly in this setting if we take $\sigma=0$ in \eqref{noisecoefficientofQ} and in \eqref{equationQlambda}. Theorem \ref{theorem3.1good} and Corollary \ref{convergenceprocess}, \ref{corollary1.4} hold true, with the same proof.  Moreover, the result of Theorem \ref{shapetheorem} now becomes, upon taking $\sigma=0$,  
    $$ \left(\mu,n|\psi^\mu(\lfloor nt\rfloor)|^2 dt\right)
    \Rightarrow (E,1_{[0,1]}(t)dt).$$ That is, the scaled eigenfunctions are uniformly spread out. 
\end{corollary}
This shape of eigenvalue result is consistent with the result of \cite{nakano2022shape}, Theorem 1.1 which considers a model in continuous time. 
    
\begin{remark}
    For potentials vanishing slower than the critical rate $\frac{1}{2}$, Poisson statistics is expected and the rescaled eigenvector should converge to Dirac measure with a uniformly distributed center. However, a proof of these claims are currently only available in the continuous model with randomness from a Brownian motion, see \cite{kotani2017poisson}, \cite{nakano2022shape}.
\end{remark}

\subsection{Properties of the point process}

The point process $\eta Sch$ has many properties that are similar to the point process $Sch_\tau$ introduced in \cite{kritchevski2011scaling}, which correspond to the vanishing model, i.e. $\eta=\frac{1}{2}$. We prove that the point process $\eta Sch$ has the following properties:

\begin{proposition}\label{prop1.5}(Eigenvalue repulsion).
Fix $\eta\in(0,\frac{1}{2}]$. For any $\mu\in\mathbb{R}$ and $\epsilon>0$,
    \begin{equation}\label{largegap3.15}
        \mathbb{P}\{\eta Sch[\mu,\mu+\epsilon]\geq 2\}\leq 4\exp(-C_{\eta,\sigma,\rho}(\log(1/\epsilon)-d(1))^2),
    \end{equation}
    for some constant $C_{\eta,\sigma,\rho}>0$, and where $d(1)=1+\int_0^1 \frac{\sigma^2\rho^2}{(1-t)^{1-2\eta}}dt$. We implicitly require that the expression holds when the value in the square is positive. 
\end{proposition}
For GOE, GUE and GSE, the probability on the left hand side of \eqref{largegap3.15} has order $\epsilon^{2+\beta}$ with $\beta=1,2,4$, which is much larger than the right hand side of \eqref{largegap3.15}. We also have:

    \begin{proposition}\label{prop1.6}
  ( Large gap probability). Fix $\eta\in(0,\frac{1}{2}]$. The probability that the point process $\eta Sch$ has a gap of length $\lambda$ is
$$\mathbb{P}(\eta Sch[0,\lambda]=0)=\exp(-c_\eta \frac{\lambda^2}{4\sigma^2\rho^2}(1+o(1)),$$
   where $o(1)\to 0$ for fixed $\eta,\rho,\sigma$ as $\lambda\to \infty$, and $c_\eta$ is some constant that depends only on $\eta$.
\end{proposition}

We have seen that the $\eta Sch$ point process has many similar features to the $Sch_\tau$ process and $\text{Sine}_\beta$ process, but they have qualitatively different behaviors at finer scales. When we consider central limit theorem of  $\eta Sch$, we can currently only prove the following crude upper bound. We leave the investigation of Gaussian fluctuations for $\eta Sch$ to future work.

   \begin{proposition}\label{prop1.7} Assume $\eta\in(0,\frac{1}{2}]$.
    Consider any increasing function $f:(0,\infty)\to\mathbb{R}_+$ with $\lim_{t\to\infty}f(t)=\infty$. Then as $\lambda\to\infty$, we have the convergence
$$\frac{1}{f(\lambda)}\left(\eta Sch[0,\lambda]-\frac{\lambda}{2\pi}\right)\to 0$$ in probability.
\end{proposition}

When $\eta=0$, i.e. we consider the critical decaying model, it is proved in \cite{kritchevski2011scaling}, Theorem 17 that for the corresponding scaling limit, the $Sine_\beta$ process, we have Gaussian fluctuations
$$\frac{1}{\sqrt{\log(\lambda)}}\left(Sine_\beta[0,\lambda]-\frac{\lambda}{2\pi}\right)\Rightarrow \mathcal{N}(0,\frac{2}{\beta\pi^2}).$$

When $\eta=\frac{1}{2},$ i.e. we consider the critical vanishing model, the scaling limit $Sch_\tau$ process satisfies a more refined form of central limit theorem ( \cite{kritchevski2011scaling}, Theorem 13)
\begin{equation}
    Sch_\tau[0,2\pi k+\theta]-k\Rightarrow \lfloor \frac{\xi_0+\xi_2+\theta}{2\pi}\rfloor-\lfloor \frac{\xi_0+\xi_1}{2\pi}\rfloor
\end{equation}
with $\xi_0,\xi_1,\xi_2$ independent normals with mean $0$ and variance $\tau,\tau/2,\tau/2$. It seems hard to generalize this argument to the general case $\eta\in(0,\frac{1}{2})$. 

\subsection{Plan of the paper}
This paper is organized as follows. In Section \ref{section2}, \ref{section3} and \ref{section4} we respectively prove the scaling limits of transfer matrices, Theorem \ref{convergence2.1}, the tightness estimates, Theorem \ref{theorem3.1good}, and the process level limit, Corollary \ref{convergenceprocess}. In Section \ref{section5} we prove the shape theorem of eigenfunctions, Theorem \ref{shapetheorem}. In Section \ref{section6} we prove the properties of the point process $\eta Sch$, Proposition \ref{prop1.5}, \ref{prop1.6}, \ref{prop1.7}.

\section{Transfer matrices evolution}\label{section2}
In this section we prove Theorem \ref{convergence2.1}.
\begin{proof}
We fix the value of $\lambda\in\mathbb{C}$ and remove it from the notation. From the iteration $M_\ell=T(E+\epsilon_\ell)M_{\ell-1}$ we may write
\begin{equation}
    Q_\ell =T^{-\ell}(E)\begin{pmatrix} 1&\epsilon_\ell\\0&1\end{pmatrix} T^\ell(E)Q_{\ell-1}. 
\end{equation}
Therefore $Q_\ell,0\leq\ell\leq n$ defines a Markov process with $Q_0=I$. Rewriting $T$ for $T(E)$, we use the following matrices to diagonalize $T$:
$$
T^{-\ell}\begin{pmatrix}   0&\epsilon_\ell\\0&0\end{pmatrix}T^\ell=\frac{i\rho\epsilon_\ell}{2} Z O_\ell Z^{-1},\quad O_\ell=\begin{pmatrix}
    1&-z^{2\ell}\\ \bar{z}^{2\ell}&-1
\end{pmatrix},$$from this we see that $X_\ell:=Z^{-1}Q_\ell Z$ becomes a Markov process with $X_0=I$ with increments
\begin{equation}\label{xelldefine}
    X_\ell =X_{\ell-1}+U_\ell X_{\ell-1},\quad U_\ell=i\rho\epsilon_\ell O_\ell/2,
\end{equation}

It suffices to show the Markov process $X_\ell^n$ converges to the limiting SDE as stated in the theorem. Since $z$ is complex with $|z|=1$, $U_\ell X_{\ell}$ does not converge to a limit. But for $K$ large, $X_{\ell+K}-X_\ell=\sum_{j=1}^K U_{\ell+j}X_{\ell+j-1}$ should converge to a limiting SDE, noting that the sum $\sum_{\ell=1}^K z^{2\ell}$ is uniformly upper bounded however large $K$ is. A rigorous proof uses the convergence result of discrete time Markov chains to SDEs, stated in \cite{kritchevski2011scaling}, Proposition 26. We briefly recall the statement of this proposition: consider the Markov chain 
$$(X_\ell^n\in\mathbb{R}^d,\ell=0,\cdots,\lfloor nT\rfloor),$$
and let $Y_\ell^n(x)$ denote the expectation of $X_{\ell+1}^n-X_\ell^n$ conditioned on $X_\ell^n=x$, and consider 
$$b^n(t,x):=n\mathbb{E}Y^n_{\lfloor nt\rfloor}(x),\quad a^n(t,x):=n\mathbb{E} Y^n_{\lfloor nt\rfloor}(x)  Y^n_{\lfloor nt\rfloor}(x)^T.$$
Assume $a^n$ and $b^n$ converges to some $C^2$ functions $a$ and $b$, in the sense that
\begin{equation}\label{convergencecondition}
    \sup_{0\leq t\leq T,|x|\leq R}\|\int_0^t (a^n(s,x)-a(s,x))ds\|+\sup_{0\leq t\leq T,|x|\leq R}\|\int_0^t (b^n(s,x)-b(s,x))ds\|\to 0,
\end{equation}
and that some mild integrability and regularity assumptions are satisfied, then $(X_{\lfloor nT\rfloor}^n,0\leq t\leq T)$ converges in distribution in $D([0,T])$ to the stochastic differential equation
\begin{equation}
    dX(t)=b(t,X(t))dt+g(t,X(t))d\mathcal{B}(t),\quad X(0)=X_0,
\end{equation}
in which $g$ is a $C^2$ function with $g(t,x)g(t,x)^T=a(t,x)$ and $\mathcal{B}(t)$ a $d$-dimensional Brownian motion. In our setting of mixed vanishing-decaying potentials, the coefficient $a(t,x)$ tends to infinity as $t\to 1$ and is thus not globally Lipschitz continuous. This is not an issue as we may apply the Proposition on $[0,1-\epsilon]$ for any $\epsilon>0$.

Now we turn back to the proof of the theorem. Conditioning on $X_\ell^n=x$, $X_{\ell+1}^n-X_\ell^n$ equals
$$Y_\ell^n(x):=\left(\frac{i\lambda}{2n}-\frac{i\sigma\rho \omega_{\ell+1}}{2n^\eta(n+1-\ell)^{\frac{1}{2}-\eta}}\right) \begin{pmatrix} 1&z^{2(\ell+1)}\\-\bar{z}^{2(\ell+1)} & -1\end{pmatrix}x.$$

Then we set $b^n(t,x)$ to be
$$\frac{i\lambda}{2}\begin{pmatrix}1&z^{2(\ell+1)}\\-\bar{z}^{2(\ell+1)}&-1\end{pmatrix}x,\quad  \ell=\lfloor nt\rfloor.$$

Now we estimate the variance $a^n(t,x).$ To simplify the analysis, we identify the $2\times 2$ complex matrix $X_\ell^n$ with a vector in $\mathbb{R}^8$, and thus identify $b^n(t,x)$ with a vector in $\mathbb{R}^8$ and identify $a^n(t,x)$ as a $8\times 8$ real matrix. Then each entry of $a^n(t,x)$ should be linearly composed of $A_jA_k$, $A_j\bar{A}_k$, $\bar{A}_j\bar{A}_k$ for $j,k\in\{1,2\}^2$, given
$$A=-\frac{i\sigma\rho}{2}(\frac{n}{n+1-\ell})^{1-2\eta}\begin{pmatrix}1&z^{2(\ell+1)}\\-\bar{z}^{2(\ell+1)}& -1\end{pmatrix}x,\quad \ell=\lfloor nt\rfloor.$$

Since $\frac{1}{n}\sup_{1\leq \ell\leq n} \sum_{j=1}^\ell \frac{z^{2j}}{(1+\frac{1-j}{n})^{1-2\eta}}\to 0$ and $\frac{1}{n}\sup_{1\leq \ell\leq n} \sum_{j=1}^\ell \frac{z^{4j}}{(1+\frac{1-j}{n})^{1-2\eta}}\to 0$ as $n$ tends to infinity (we check the first estimate as the second estimate is analogous: we use the idea of Abel summation where we define $a_{j,n}=(\frac{n+1-j}{n})^{2\eta-1}$ and define $A_j=\sum_{k=1}^j z^{2k}$. Then the summation can be rewritten as $\frac{1}{n}(a_{\ell,n} A_\ell-\sum_{j=1}^{\ell-1} (a_{j+1,n}-a_{j,n})A_j)$. Then one can check that $|A_j|\leq\frac{2}{{|z^2-1
|}}$ and $\sum_{j=1}^{\ell-1}|a_{j+1,n}-a_{j,n}|=O(n^{1-2\eta})$, the latter can be checked from the elementary estimate $|a_{j+1,n}-a_{j,n}|=O(\frac{1}{n}\cdot (\frac{n+1-j}{n})^{2\eta-2})$.) Then the limit $a(t,x)$ and $b(t,x)$ are obtained from $a^n$ and $b^n$ by taking $n\to\infty$ and replacing all the powers of $z$ and $\bar{z}$ by $0$. The convergence \eqref{convergencecondition} then follows.

Now we can verify that the limiting SDE should be
\begin{equation}\label{noisecoefficient}
    dX=\begin{pmatrix} i\lambda/2&0\\0&-i\lambda/2\end{pmatrix} X dt+\frac{\sigma\rho}{2}\frac{1}{(1-t)^{\frac{1}{2}-\eta}}\begin{pmatrix}
        id\mathcal{B}& d\mathcal{W}\\d\bar{\mathcal{W}}&-id\mathcal{B} \end{pmatrix} X,\quad X(0)=I.
\end{equation}
The drift of this SDE agrees with $b(t,x)$. For the variance, consider the random vector 
$$-\frac{\sigma\rho}{2}\begin{pmatrix}iB&W\\\bar{W}&-iB\end{pmatrix} x$$ with $B$ and $W$ having standard real and complex normal distribution. Thus the variance of the SDE agrees with that of $a(t,x)$. This finishes the proof.

In the case $E=0$, we take $z=i$ so $z^{4j}=1$ and the SDE limit is
\begin{equation}
    dX=\begin{pmatrix} i\lambda/2&0\\0&-i\lambda/2\end{pmatrix} X dt+\frac{\sigma\rho}{2}\frac{1}{(1-t)^{\frac{1}{2}-\eta}}\begin{pmatrix}
        id\mathcal{B}_1& id\mathcal{B}_2\\-id\mathcal{B}_2&-id\mathcal{B}_1 \end{pmatrix} X,\quad X(0)=I.
\end{equation}
\end{proof}

\section{Tightness estimates}\label{section3}

In this section we prove Theorem \ref{theorem3.1good}.

\begin{proof}
    We essentially follow the steps of \cite{kritchevski2011scaling}, Theorem 1 and \cite{rifkind2018eigenvectors}, Lemma 4.1. The main adaptation follows from the following: for any $n>0,$ 
    \begin{equation}
        \label{usefulhaha}
   \prod_{i=1}^n(1+\frac{1}{n^{2\eta}(n+1-i)^{1-2\eta}})\leq C_\eta<\infty\end{equation}
    for some constant $C_\eta>0$. Indeed, we take the logarithm of the left hand side and use
    $$\log(1+\frac{1}{n^{2\eta}(n+1-i)^{1-2\eta}})\leq \frac{1}{n^{2\eta}(n+1-i)^{1-2\eta}}$$
    and
    \begin{equation}
        \label{takingintoaccount}
    \sum_{i=1}^n \frac{1}{i^{1-2\eta}}\leq 1+\int_1^n \frac{1}{x^{1-2\eta}}dx= \frac{n^{2\eta}-1}{2\eta}+1.\end{equation}

    The proof of this theorem follows from the following claim: we can find a function $f$ continuous and bounded on $(-2,2)$ so that
    \begin{equation}
        \label{usefulmaximal}
   \sup_n \max_{0\leq\ell\leq n}\mathbb{E}\|M_n(E,\ell)-I\|^3<f(E),\quad E\in(-2,2).\end{equation} Once this is shown, \eqref{eigenvectorupper} follows from the upper bound on number of eigenvalues by transfer matrices norms, see Theorem 6.1 of \cite{rifkind2018eigenvectors}. Estimate \eqref{eigenpairups} also follows from standard, deterministic  arguments.

Now we show \eqref{usefulmaximal}. By the diagonalization \eqref{Ediagonalized}, we see that  $\|T^{-\ell}(E)\|\leq f(E)<\infty$ uniformly in $\ell\geq 1,$ so it suffices to prove 
 \begin{equation}
   \sup_n \max_{0\leq\ell\leq n}\mathbb{E}\|X_\ell-I\|^3<f(E),\quad E\in(-2,2),\end{equation}
where $X_\ell$ is defined in \eqref{xelldefine} and is a Markov chain.

By the recursion formula \eqref{xelldefine}, $X_k$ forms a martingale. We use the Burkholder-Davis-Gundy inequality and Doob’s Decomposition to expand the squares:
we have for any $\ell\geq 1,$
$$\begin{aligned}\mathbb{E}&\max_{k\leq\ell}\|X_k-I\|^3\leq c_2\mathbb{E}(\sum_{k=1}^\ell \mathbb{E}[\|X_k-X_{k-1}\|^2\mid\mathcal{F}_{k-1}])^{3/2}
&\\&\leq c\rho(E)^3\frac{1}{n}\mathbb{E}\sum_{k=1}^\ell \|X_{k-1}\|^3(\frac{n}{n+1-
k})^{1-2\eta},\end{aligned}$$ where in the last line we used the variance of the potentials $\epsilon_\ell$ and used Jensen's inequality.
Now using the inequality  $\|A+B\|^p\leq 2^p(\|A\|^p+\|B\|^p),$
we further estimate, taking into account \eqref{takingintoaccount},
\begin{equation}\label{mainests}\begin{aligned}
\mathbb{E}\max_{k\leq\ell}\|X_k-I\|^3&\leq c\rho(E)^3\left[\frac{1}{n}\sum_{k=1}^\ell (\frac{n}{n+1-
k})^{1-2\eta} +\frac{1}{n}\sum_{k=1}^\ell \|X_{k-1}-I\|^3 (\frac{n}{n+1-
k})^{1-2\eta}\right]\\
&\leq c'\rho(E)^3(1+\frac{S_{\ell-1}}{n}),
\end{aligned}\end{equation}
where we define the weighted sum $S_{\ell}=\sum_{k=1}^\ell \mathbb{E}\|X_k-I\|^3(\frac{n}{n-k})^{1-2\eta}$.

Now rewriting the previous expression we have $$(\frac{n-\ell}{n})^{1-2\eta}(S_\ell-S_{\ell-1})=\mathbb{E}\|X_\ell-I\|^3\leq c'\rho(E)^3(1+\frac{S_{\ell-1}}{n}).$$

Considering $R_\ell=1+\frac{S_\ell}{n}$, we see that 
$$R_\ell-  R_{\ell-1}\leq n^{-2\eta} (n-\ell)^{2\eta-1}c'\rho(E)^3R_{\ell-1}.$$ Rearranging the expression, we have 
$$
R_\ell\leq (1+c'\rho(E)^3n^{-2\eta}(n-\ell)^{2\eta-1})R_{\ell-1}.
$$
Now we complete the proof: from \eqref{usefulhaha} we conclude that $R_\ell\leq C_\eta<\infty$ uniformly in $\ell=1,\cdots,n-1$ and in $n\geq 0$. Then by definition of $R_\ell$ and \eqref{mainests} we deduce 
$$\mathbb{E}\max_{k\leq\ell}\|X_k-I\|^3\leq c_\eta\rho(E)^3<\infty.$$
    
\end{proof}

\section{Process level convergence and SDEs}\label{section4}

The proof of corollary \ref{convergenceprocess} is similar to Corollary 3 of \cite{kritchevski2011scaling} and Theorem 2.1 of \cite{rifkind2018eigenvectors}, so we only give a sketch:
\begin{proof}(Sketch) 
For the proof, 
three main arguments are used: (1) the uniform upper bound of transfer matrices with high probability, which in this paper is derived in Theorem \ref{theorem3.1good}; (2) the convergence of $Q_n$ to $Q$ for fixed and finite-dimensional $\lambda$, which in this paper is derived in Theorem \eqref{convergence2.1}, and (3) properties of analytic functions, as our transfer matrices are analytic in $\lambda$. The main difference from Theorem 2.1 of \cite{rifkind2018eigenvectors} is we only have convergence for the uniform topology on $\mathbb{C}\times[0,1-\epsilon]$ rather than $\mathbb{C}\times[0,1]$. This is because the same restriction has been posed for the convergence to $Q^\lambda$ in Theorem \ref{convergence2.1}.

More precisely, to show $M_{n_j}^\lambda$ converges in distribution to $\widetilde{T}Q^\lambda(1)$, we use tightness (Theorem \ref{theorem3.1good}) to extract a further subsequence of integers $\{n_{j_k}\}$ and a limit $\widetilde{Q}^\lambda$ so that $Q_{n_{j_k}}^\lambda\to_{law} \widetilde{Q}^\lambda$. Since $T^{n_j}\to \widetilde{T}$, this implies $M^\lambda_{n_{j_k}}\to \widetilde{T}\widetilde{Q}^\lambda$. To identify the limit point $\widetilde{Q}^\lambda$ with $\lim_{t\uparrow 1}Q^\lambda(t)$, we use three things: (A) Theorem \eqref{convergence2.1} which gives the convergence of the process $Q^\lambda_{n_{j_k}t}$ to $Q^\lambda(t)$ at any time $t<1$, (B) the fact that $\lim_{t\uparrow 1}Q^\lambda(t)$ exists almost surely: this follows from computing the quadratic variation of the SDE solved by $Q^\lambda(t)$, and (C) the fact that $\lim_{t\uparrow 1}Q^\lambda_{n_{j_k}t}$ converges to a unique limit $Q^\lambda_{n_{j_k}}$ almost surely: this fact can be checked by working through the proof of Theorem \ref{theorem3.1good} where we compute the martingale difference and estimate the moment of $Q_{n_{j_k}t}^\lambda-Q_{n_{j_k}}^\lambda$: the latter computation is very similar to the tightness proof and is therefore omitted for simplicity.
\end{proof}
Then we characterize the evolution of $(r^\lambda,\theta^\lambda)$ in terms of SDEs with parameter $\lambda$. Recall that $(r^\lambda,\theta^\lambda)$ are defined in \eqref{rlambdathetalambda}.

\begin{proposition}\label{whatisthesde?}
    $r^\lambda$ and $\theta^\lambda$ satisfies the following set of stochastic differential equations:

\begin{equation}\label{thetalambda}
    d\theta^\lambda(t)=\lambda dt+\frac{\sigma\rho}{(1-t)^{\frac{1}{2}-\eta}}[d\mathcal{B}+\text{Im}[e^{-i\theta^\lambda(t)}d\mathcal{W}]],\quad \theta^\lambda(0)=0,
\end{equation}

\begin{equation}
dr^\lambda(t)=\frac{\sigma^2\rho^2}{4(1-t)^{1-2\eta}}dt+\frac{\sigma\rho}{(1-t)^{\frac{1}{2}-\eta}}\text{Re}[e^{-i\theta^\lambda(t)}d\mathcal{W}],\quad r^\lambda(0)=0,
\end{equation}
given $\mathcal{B}$ and $\mathcal{W}$ independent standard real and complex Brownian motions.

Defining $\phi^\lambda(t):=\frac{\partial\theta^\lambda(t)}{\partial\lambda}$, then $\phi^\lambda$ solves
\begin{equation}\label{4.3.}
d\phi^\lambda(t)=dt-\frac{\sigma\rho}{(1-t)^{\frac{1}{2}-\eta}}\text{Re}[e^{-i\theta^\lambda(t)}d\mathcal{W}]\phi^\lambda(t).
\end{equation}
\end{proposition}

\begin{proof} Consider $X(\lambda,t)=Z^{-1}Q(\lambda,t)$, then from \eqref{noisecoefficientofQ} we see that
     $X_{11}$ solves the following SDE
    $$
dX_{11}(\lambda,t)=\frac{i\lambda}{2}X_{11}(\lambda,t)dt+\frac{\sigma\rho}{2(1-t)^{\frac{1}{2}-\eta}}[iX_{11}(\lambda,t)d\mathcal{B}+X_{21}(\lambda,t)d\mathcal{W}].
    $$
    From \eqref{nicestory}, we see that $q^\lambda(t)=X(\lambda,t)_{11}$ and that $\bar{q}^\lambda(t)=X(\lambda,t)_{(21)}$ so $q$ solves 
    $$dq=\frac{i\lambda}{2}qdt+\frac{\sigma\rho}{2(1-t)^{\frac{1}{2}-\eta}}(iqd\mathcal{B}+\bar{q}d\mathcal{W}),\quad q(0)=1.$$
    From Ito's formula, regarding $q$ as a two-dimensional real vector,
    $$\begin{aligned}
d\log q=&\frac{dq}{q}-\frac{1}{2}\frac{(dq)^2}{q^2}\\
=&\frac{i \lambda}{2}dt +\frac{\sigma\rho}{2(1-t)^{\frac{1}{2}-\eta}}[id\mathcal{B}+\frac{\bar{q}}{q}d\mathcal{W}]+\frac{\sigma^2\rho^2}{8(1-t)^{1-2\eta}}dt.
 \end{aligned}   $$

 From the defining relation $r=2\text{Re}\log q$ and $\theta=2\text{Im}\log q$, we see they satisfy the following SDEs
$$ dr=\text{Re}(\frac{\bar{q}}{q}d\mathcal{W})\frac{\sigma\rho}{(1-t)^{\frac{1}{2}-\eta}}+\frac{\sigma^2\rho^2}{4(1-t)^{1-2\eta}}dt$$
 and 
 $$d\theta=\lambda dt+\frac{\sigma\rho}{(1-t)^{\frac{1}{2}-\eta}}[d\mathcal{B}+\text{Im}(\frac{\bar{q}}{q}d\mathcal{W})].$$
 Finally use $\frac{\bar{q}}{q}=\exp(-i\theta).$ The proof of \eqref{4.3.} follows from differentiating both sides of \eqref{thetalambda} with respect to $\lambda$.
\end{proof}

Then the proof of Corollary \ref{corollary1.4} directly follows.

\begin{proof}[\proofname\ of Corollary \ref{corollary1.4}]For this proof only, we use a slightly different diagonalization of $T(E)$: we take $T(E)=\hat{Z}D\hat{Z}^{-1}$ with 
\begin{equation}
    D=\begin{pmatrix}
        \bar{z}&0\\0&z
\end{pmatrix},\quad \hat{Z}=\begin{pmatrix}
    \bar{z}&-z\\1&-1
\end{pmatrix},\quad z=\frac{E}{2}+i\sqrt{1-\frac{E^2}{4}}.
\end{equation}

We first assume that $z^{n_j+1}$ converges to $e^{i\theta}$.
Then $T^{n_j}(E)$ converges to a matrix $\tilde{T}$ with
$$
\tilde{T}\hat{Z}=\hat{Z}\lim\begin{pmatrix}\bar{z}^{n_j}&0\\0&z^{n_j}\end{pmatrix}=\begin{pmatrix}
\bar{z}&-z\\1&-1
\end{pmatrix}
\begin{pmatrix}
ze^{-i\theta}&0\\0&\bar{z}e^{i\theta}
\end{pmatrix}=\begin{pmatrix}
    e^{-i\theta}&-e^{i\theta}\\ze^{-i\theta}&-\bar{z}e^{i\theta}
\end{pmatrix}.
$$
By Corollary \ref{convergenceprocess}, item (2), we have to identify the zeros of 
$$
[\tilde{T}Q^\lambda(1)]_{11}=[\tilde{T}\hat{Z}\tilde{X}^\lambda(1)\hat{Z}^{-1}]_{11}=[\tilde{T}\hat{Z}X^\lambda(1)]_{11}.
$$Here $\tilde{X}^\lambda$ solves the SDE \eqref{noisecoefficient} (where we use that $Z=\hat{Z}\frac{i\rho(E)}{2}$ so $Z\tilde{X}^\lambda(1)Z^{-1}=\hat{Z}\tilde{X}^\lambda(1)\hat{Z}^{-1}$) and thus by linearity, $X^\lambda:=\tilde{X}^\lambda \hat{Z}^{-1}$ solves the same SDE with initial condition $X^\lambda(0)=\hat{Z}^{-1}$.
We can verify that for a real $\lambda$ and any $t>0$, the solution $X^\lambda(t)$ has the form $\begin{bmatrix} a&b\\ -\bar{a}&-\bar{b}
\end{bmatrix}$, this is because the initial condition has this form and the SDE preserves solutions of this form. 
Then we can check that 
\begin{equation}
    [\tilde{T}\hat{Z}X^\lambda(1)]_{11}=e^{-i\theta}X^\lambda(1)_{11}-e^{i\theta}X^\lambda(1)_{21}=2\Re[e^{-i\theta}X^\lambda(1)_{11}].
\end{equation}

Next we rewrite the expression of the SDE solved by matrix entries:
\begin{equation}
   2dX_{11}=i\lambda X_{11}dt+\frac{\sigma\rho}{(1-t)^{\frac{1}{2}-\eta}}[i X_{11}d\mathcal{B}+\bar{X}_{11}d\mathcal{W}],\quad X_{11}(0)=i\rho/2, 
\end{equation}
\begin{equation}
    2dX_{12}=i\lambda X_{12}dt+\frac{\sigma\rho}{(1-t)^{\frac{1}{2}-\eta}}[iX_{12}d\mathcal{B}+\bar{X}_{12}d\mathcal{W}],\quad X_{12}(0)=-iz\rho/2,
\end{equation} Via applying Itô formula we can verify that 
$$
d\det \tilde{X}^\lambda=0,\quad \text{and}\quad 2i\Im[X_{11},\tilde{X}_{   12}]=\det X^\lambda(t)=\det Z^{-1}\neq 0.
$$Thus $X_{11}^\lambda(t)$ is never equal to 0.
   We define the phase function $\varphi^\lambda(t)$ via
    $$
e^{i\varphi^\lambda(t)}=\frac{iX_{11}^\lambda(t)}{\overline{iX_{11}^\lambda(t)}},\quad \varphi^\lambda(0)=0,$$ then the phase function is well-defined and we can apply Itô's formula and get that $\varphi^\lambda(t)$ satisfies the SDE \eqref{dphilambdat} with $-d\mathcal{W}$ replaced by $d\mathcal{W}$.

Moreover, we have that the zeros of $[\tilde{T}\hat{Z}X^\lambda(1)]_{11}$ correspond to the solutions to $$
\Re[e^{-i\theta-i\pi/2}iX^\lambda(1)_{11}]=0,
$$ which is equivalently $-\theta-\pi/2+\varphi^\lambda(1)/2\in\pi\mathbb{Z}$. To sum up, we have verified that if $z^{n_j+1}\to e^{i\theta}$, then we have the convergence in law of  
$$
\Lambda_{n_j}\Rightarrow  \{\lambda:\varphi^{\lambda}(1)\in 2\theta+\pi+2\pi\mathbb{Z}\}=\eta Sch^{2\theta+\pi}.$$
Since $2\arg(z^{n+1})-\arg(z^{2n+2})$ is either 0 or $2\pi$ and $\operatorname\eta{Sch}^{2\theta+\pi}=^d \eta{Sch}+2\theta+\pi$, this verifies the statement of the corollary.
\end{proof}

\section{Shape of eigenfunctions}\label{section5}

We separate the proof of Theorem \ref{shapetheorem} into three subsections. The first contains some preparatory lemmas. The second proves scaling limit at microscopic scales of joint eigenvalue eigenvector pairs near a bulk energy $E$. In the third subsection we upgrade the microscopic scaling limit into global scaling limit.

\subsection{Technical preparations}
\begin{lemma}\label{lema5.1}
Given $0<|E|<2$, consider $\boldsymbol{m}_n^\lambda,\boldsymbol{q}^\lambda$ to be measures on $[0,1]$ having densities
$$d\boldsymbol{m}_n^\lambda(t)=|((2/\rho(E))M_{n,E}(\lambda,\lfloor nt\rfloor)_{11}|^2dt
$$
$$d\boldsymbol{q}^\lambda(t)=|q^\lambda(t)|^2dt.
$$
Given any sequence $n_j$ such that $z^{n_j+1}\to z'$, we have the convergence in law
$$
\{(\lambda,\boldsymbol{m}_n^\lambda):\lambda\in \Lambda_{n_j,E}\}\Rightarrow \{(\lambda,2\boldsymbol{q}^\lambda):\lambda\in \eta Sch^{2\arg z'}\}.
$$
\end{lemma}

 The proof of this lemma is the same as in \cite{rifkind2018eigenvectors}, Lemma 3.4, so we omit the details.

\begin{lemma}\label{lemma3.6!}
    Given $U$ a uniform distribution on $[0,2\pi]$, we have for any $\phi\in\mathbb{R}$,
$$
\{(\lambda+U,\boldsymbol{q}^\lambda):\lambda\in\eta Sch^\phi\}=^d \{(\lambda,\boldsymbol{q}^\lambda):\lambda\in \eta Sch^*\}.$$
\end{lemma}

\begin{proof}
Since $r^\lambda=2\log|q^\lambda|,$ we will show 
$$
\{(\lambda+U,r^\lambda):\lambda\in \eta Sch^\phi\}=^d \{(\lambda,r^\lambda):\lambda\in \eta Sch^*.
$$
Recall that $\theta^\lambda$ and $r^\lambda$ solve the SDE
\begin{equation}
    d\theta^\lambda(t)=\lambda dt+\frac{\sigma\rho}{(1-t)^{\frac{1}{2}-\eta}}[d\mathcal{B}+\text{Im}[e^{-i\theta^\lambda(t)}d\mathcal{W}]],\quad \theta^\lambda(0)=0,
\end{equation}

\begin{equation}
dr^\lambda(t)=\frac{\sigma^2\rho^2}{4(1-t)^{1-2\eta}}dt+\frac{\sigma\rho}{(1-t)^{\frac{1}{2}-\eta}}\text{Re}[e^{-i\theta^\lambda(t)}d\mathcal{W}],\quad r^\lambda(0)=0.
\end{equation}
To understand the effect of shifting by $u$, define $\tilde{\theta}^\lambda(t):=\theta^{\lambda-u}(t)+ut$ and $\tilde{r}^\lambda(t):=r^{\lambda-u}(t)$, and notice that $\tilde{\theta}$ and $\tilde{r}$ solve the same SDE with $d\widetilde{\mathcal{B}}=d\mathcal{B}$ and $d\widetilde{\mathcal{W}}=e^{iut}d\mathcal{W}.$ 
    Now since $\widetilde{\mathcal{W}}$ has the same law as $\mathcal{W},$ by uniqueness of solutions to the SDEs, $(r^\lambda,\theta^\lambda)$ has the same law as $(\tilde{r}^\lambda,\tilde{\theta}^\lambda)$. Since $\theta^{\lambda-u}(1)=\tilde{\theta}(1)-u$, we have
$$ \begin{aligned}  \{(\lambda+u,r^\lambda):\lambda\in \eta Sch^\phi\}
&=\{\lambda,r^{\lambda-u}:\lambda\in\eta Sch^\phi+u\}
\\ &=\{(\lambda,\tilde{r}^\lambda):\tilde{\theta}^\lambda(1)\in 2\pi \mathbb{Z}+\phi+u\}
\\&=^d \{(\lambda,r^\lambda):\lambda\in \eta Sch^{\phi+u}\}. \end{aligned}
$$
    For $U$ uniform on $[0,2\pi]$, $U+\phi$ is also uniform on $[0,2\pi]$. This finishes the proof.
\end{proof}

\subsection{Proof of local joint scaling limit}

\begin{lemma}\label{lemma3.76}
    Given $\mathcal{B}$ a standard Brownian motion with $\mathcal{B}_0=0$ and $U$ uniform on $[0,1]$, define $f^u(t)=(u-|u-t|)$. We have for any $G\in \mathcal{C}_b(\mathbb{R}\times C[0,1]),$
    $$
\mathbb{E}\sum_{\lambda\in\eta Sch^*}G(\lambda,|q^\lambda|^2)=\frac{1}{2\pi}\int d\lambda\mathbb{E}\left[G\left(\lambda,\exp(\mathcal{B}_{v(\cdot)}+\frac{f^{v(U)}(v(\cdot))}{2})\right)\right],$$
\end{lemma}
where the function $v:[0,1]\to\mathbb{R}_+$ is defined by  \begin{equation}\label{whatisvt?}v(t)=\frac{\sigma^2\rho^2}{4\eta}[1-(1-t)^{2\eta}].\end{equation}
\begin{proof}
    By definition, $\eta Sch^*=\{\lambda:\theta^\lambda(1)\in 2\pi\mathbb{Z}+U\}$ and $U$ uniform on $[0,2\pi]$. Take expectation over $U$,
    $$\begin{aligned}
\mathbb{E}\sum_{\lambda\in\eta Sch^*}G(\lambda,r^\lambda)&=\frac{1}{2\pi}\mathbb{E}\int_0^{2\pi} du \sum_{\lambda:\theta^\lambda(1)\in 2\pi\mathbb{Z}+u}G(\lambda,r^\lambda)\\&=\frac{1}{2\pi}\mathbb{E}\int_{-\infty}^\infty du \sum_{\lambda:\theta^\lambda(1)=u}G(\lambda,r^\lambda).\end{aligned}   $$

Now we know $\theta^\lambda(1)$ is a.s. analytic in $\lambda$ (This can be checked as follows: take a sequence $t_n\uparrow 1$, we first check $\theta^\lambda(t_n)$ is analytic in $\lambda$, which can be checked by differentiating the SDE solved by $\theta^\lambda(t_n)$ by $x:=\Re\lambda$ and $y:=\Im\lambda$, and check the Cauchy-Riemann equation $\partial_x \theta^\lambda(t_n)=i\partial_y \theta^\lambda(t_n)$ is satisfied because both of them are solutions to the same SDE. Then we can show that almost surely, $\theta^\lambda(t_n)$ converges locally uniformly in $\lambda\in\mathbb{C}$ to $\theta^\lambda(1)$ by computing the quadratic variation of $\frac{d\mathcal{W}}{(1-t)^{\frac{1}{2}-\eta}}$, and the latter limit is a.s. analytic by Montel's theorem.), and $r^\lambda$ is a.s. continuous in $\lambda$. By the co-area formula, 

\begin{equation}
    \frac{1}{2\pi}\mathbb{E}\int_{-\infty}^\infty du \sum_{\lambda:\theta^\lambda(1)=u}G(\lambda,r^\lambda)=\frac{1}{2\pi}\int_{-\infty}^\infty d\lambda \mathbb{E}[G(\lambda, r^\lambda)|\frac{\partial \theta^\lambda(1)}{\partial \lambda}|].
\end{equation}
Notice that $\phi^\lambda(t):=\frac{\partial\theta^\lambda(t)}{\partial\lambda}$ solves the SDE (see \eqref{4.3.})
$$d\phi^\lambda=dt-\frac{\sigma\rho}{(1-t)^{\frac{1}{2}-\eta}}\text{Re}\left[e^{-i\theta^\lambda(t)}d\mathcal{W}\right]\phi^\lambda.$$

Recall also that $r^\lambda$ solves the SDE 
$$
dr^\lambda=\frac{\sigma^2\rho^2}{4(1-t)^{1-2\eta}}dt+\frac{\sigma\rho}{(1-t)^{\frac{1}{2}-\eta}}\text{Re}[e^{-i\theta^\lambda(t)}d\mathcal{W}],\quad r^\lambda(0)=0.$$
Since $e^{-i\theta^\lambda}d\mathcal{W}=^d d\mathcal{W}$, the distribution of $(r^\lambda,\phi^\lambda)$ is $\lambda$-independent. We may temporarily write $d\mathcal{B}=\sqrt{2}\text{Re}[e^{-i\theta^\lambda(t)}d\mathcal{W}]$ where $\mathcal{B}$ is some two-sided Brownian motion, and solve (omitting $\lambda$ in the expression),
$$\begin{aligned}
d\phi(t)=&dt-\phi dr +\frac{\sigma^2\rho^2}{4(1-t)^{1-2\eta}}\phi dt
.\end{aligned}$$
 By Itô's formula, the solution is 
 $$\phi(t)=\int_0^t du e^{r(u)-r(t)}.$$
From Fubini's theorem, 
$$\mathbb{E}\left[G(\lambda,r^\lambda)\mid \frac{\partial\theta^\lambda(1)}{\partial\lambda}|\right]=\int_0^1 du\mathbb{E}[e^{r(u)-r(1)}G(\lambda,r)].$$
Meanwhile, $r^\lambda$ solves
$$r^\lambda(t)=\int_0^t \frac{\sigma^2\rho^2}{4(1-s)^{1-2\eta}}ds +\int_0^t \frac{\sigma\rho}{(1-s)^{\frac{1}{2}-\eta}\sqrt{2}}d\mathcal{B}_s.$$

For simplicity, we introduce a time change \begin{equation}
    \label{timechange}
v(t):=\int_0^t \frac{\sigma^2\rho^2}{2(1-s)^{1-2\eta}}ds=\frac{\sigma^2\rho^2}{4\eta}[1-(1-t)^{2\eta}],\end{equation}
then $r^\lambda$ has the same law as
$$r^\lambda(v)=\frac{1}{2}v+\mathcal{Z}_v$$
for some standard two-sided Brownian motion $\mathcal{Z}$ with $\mathcal{Z}_0=0$.

Now let $\mathcal{R}$ denote the distribution of $r(v)$, $v\in[0,v(1)]$, on $\mathcal{C}[0,v(1)]$ and let $\mathcal{P}$ denote the law of $\mathcal{Z}_v$ on $\mathcal{C}[0,v(1)]$. We may characterize the law of 
$$\exp(\omega_v-\omega_1)d\mathcal{R}(\omega)$$
via Girsanov transform, as in Lemma 3.6 of \cite{rifkind2018eigenvectors}, see also \cite{nakano2022shape}, Lemma 2.3, to get that under the probability measure $\exp(\omega_u-\omega_1)d\mathcal{R}(\omega)$ on $\mathcal{C}[0,v(1)],$ a path $\omega$ has distribution $\mathcal{Z}+\frac{1}{2}f^u$ where $f^u(x)=(u-|u-x|)$ and $\mathcal{Z}$ is the standard Brownian motion. Now reworking the time change we have
$$\mathbb{E}[e^{(r(u)-r(1))}G(\lambda,r)]=\mathbb{E}[G(\lambda,\mathcal{B}_{v(\cdot)}+\frac{f^{v(u)}(v(\cdot))}{2}].$$
Taking the sum, 
$$\mathbb{E}\sum_{\lambda\in \eta Sch^*}G(\lambda,r^\lambda)=\frac{1}{2\pi}\int_{-\infty}^\infty d\lambda \int_0^1 du\mathbb{E}[G(\lambda, \mathcal{B}_{v(\cdot)}+\frac{f^{v(u)}(v(\cdot))}{2}].
$$
The proof finishes via the continuous mapping theorem applied to $f\mapsto \exp (f).$
\end{proof}

\begin{theorem}\label{newtheorem5.4}
    For fixed $0<|E|<2$, we set $\tau=\tau(E)=(\sigma\rho(E))^2$. Given $U$ uniform on $[0,2\pi]$, the point process defined on $\mathbb{R}\times \mathcal{M}([0,1])$ 
    $$ \{(n\rho(E)(\mu-E)+U,n|\psi^{\mu}(\lfloor nt\rfloor)|^2dt):\mu \text{ eigenvalue of } H_n\}$$
     converges in distribution to the point process $\mathcal{P}_E$ with intensity measure $\mu_E$. The intensity measure $\mu_E$ satisfies the following property: for any $F\in\mathcal{C}_b(\mathbb{R}\times \mathcal{M}([0,1]))$, 
$$\begin{aligned}&\int F(\lambda,\nu)d\mu_E(\lambda,\nu)\\&=\frac{1}{2\pi}\int d\lambda \mathbb{E} F\left(\lambda,\frac{\exp\left(\mathcal{Z}_{\frac{\sigma^2\rho^2}{4\eta}[(1-U)^{2\eta}-(1-t)^{2\eta}]}-\frac{\sigma^2\rho^2}{8\eta}\left|(1-U)^{2\eta}-(1-t)^{2\eta}\right|\right)dt}{\int_0^1 ds \exp\left(\mathcal{Z}_{\frac{\sigma^2\rho^2}{4\eta}[(1-U)^{2\eta}-(1-s)^{2\eta}]}-\frac{\sigma^2\rho^2}{8\eta}\left|(1-U)^{2\eta}-(1-s)^{2\eta}\right|\right)}\right),\end{aligned}$$
     
     where $\mathcal{Z}$ is a standard two-sided Brownian motion and $U$ is uniform on $[0,1]$.
\end{theorem}

\begin{proof}
    Along a subsequence $n^j$ with $z^{n_j}$ converging to $z'$, we obtain by Lemma \ref{lema5.1}
    $$
\begin{aligned} \{(\lambda+U,\boldsymbol{m}_n^\lambda):\lambda\in \Lambda_{n_j,E}\}&\Rightarrow \{(\lambda+U,2\boldsymbol{q}^\lambda):\lambda\in \eta Sch^{2\arg z'}\\
&=^d \{(\lambda,2\boldsymbol{q}):\lambda\in \eta Sch^*\}.
 \end{aligned} $$ For each subsequence we can find a sub-subsequence along which $z^{n_j}$ converges, so
 $$\{(\lambda+U,\boldsymbol{m}_n^\lambda):\lambda\in \Lambda_{n,E}\}\Rightarrow \{(\lambda,2\boldsymbol{q}):\lambda\in \eta Sch^*\}.$$

Moreover, since we may characterize the eigenvalues via
$$n|\psi^\mu(\lfloor nt\rfloor)|^2dt =\frac{d\boldsymbol{m}_n^\lambda(t)}{\boldsymbol{m}_n^\lambda([0,1])},$$
 and the functional on $\mathcal{M}([0,1]):\mu\mapsto \mu/\mu([0,1])$ is continuous away from zero, we have

$$\begin{aligned} \{(n\rho(E)(\mu-E)+U,n|\psi^\mu(\lfloor nt\rfloor)|^2 dt):\mu \text{ an  eigenvalue of }H_n\}\\
\Rightarrow \{(\lambda,\frac{\boldsymbol{q}^\lambda}{\boldsymbol{q}^\lambda([0,1])}):\lambda\in\eta Sch^{*}\}.
\end{aligned}$$
Recall that $d\boldsymbol{q}^\lambda(t)=|q^\lambda(t)|^2dt$, now we use Lemma \ref{lemma3.76}, and notice that, as the additive constants will cancel in normalization,
$$
\frac{\exp(\mathcal{B}_{v(t)}+\frac{1}{2}v(u)-\frac{1}{2}|v(u)-v(t)|)}{
\int_0^1 ds \exp(\mathcal{B}_s+\frac{1}{2}v(u)-\frac{1}{2}|v(u)-v(s)|)}=^d \frac{\exp(\mathcal{Z}_{v(t)-v(u)}-|v(t)-v(u)|/2)}{\int_0^1 ds\exp(\mathcal{Z}_{v(s)-v(u)}-|v(s)-v(u)|/2)}$$
where $\mathcal{Z}$ is a two sided Brownian motion. This finishes the proof given the explicit expression of $v(t)$ given in \eqref{whatisvt?}.\end{proof}

\subsection{Proof of global joint scaling limit}

Now we are finally in the position to prove Theorem \ref{shapetheorem}, the characterization of shape of eigenvalues.

\begin{proof} We essentially follow the proof of \cite{rifkind2018eigenvectors}, Theorem 1.1 and only give a sketch to some main arguments. Let $\theta$ be a uniform distribution on $[0,2\pi]$ and let $\boldsymbol{\psi}_n^\mu\in\mathcal{M}([0,1])$ be a measure having a density $|\psi^\mu(\lfloor nt\rfloor)|^2$. By Theorem \ref{newtheorem5.4}, the point process
$$\mathcal{P}_{E,n}=\{(n,\rho(E)(\mu-E)+\theta,n\boldsymbol{\psi}_n^\mu):\mu\in\Lambda_n\}$$
has a limit $\mathcal{P}_E$. Fix $g_1=(1-|x|)1_{|x|\leq 1}$ and $g_2\in\mathcal{C}_b(\mathbb{R}\times\mathcal{M}([0,1])),$ consider
$$G_n(E):=\sum_{\mu\in\Lambda_n}g_1(n\rho(E)(\mu-E))g_2(\mu,\boldsymbol{\psi}_n^\mu).
$$ Then given $|E|<2$ it converges in law to $G(E)$ with
\begin{equation}
    \mathbb{E}G(E)=\frac{1}{2\pi}\mathbb{E}g_2\left(E,\frac{\exp\left(\mathcal{Z}_{\frac{\sigma^2\rho^2}{4\eta}[(1-U)^{2\eta}-(1-t)^{2\eta}]}-\frac{\sigma^2\rho^2}{8\eta}\left|(1-U)^{2\eta}-(1-t)^{2\eta}\right|\right)dt}{\int_0^1 ds \exp\left(\mathcal{Z}_{\frac{\sigma^2\rho^2}{4\eta}[(1-U)^{2\eta}-(1-s)^{2\eta}]}-\frac{\sigma^2\rho^2}{8\eta}\left|(1-U)^{2\eta}-(1-s)^{2\eta}\right|\right)}\right).    
\end{equation}
The proof follows if we can show $$\int\mathbb{E}G_n(E)\rho(E)dE\to\int\mathbb{E}G(E)\rho(E)dE.$$ 

Denote by $$N_n(E)=|\{\mu\in\Lambda_n:|\mu-E|\leq 1/(n\rho(E))\}|,$$
so that $G_n(E)\leq \|g_1\|_\infty \|g_2\|_\infty N_n(E).$

By the tightness estimate in Theorem \ref{theorem3.1good}, $G_n(E)1_{|E|\leq 2-\epsilon}$ is uniformly integrable, so\begin{equation}
    \lim_{n\to\infty}\int \mathbb{E}[G_n(E)1_{|E|\leq 2-\epsilon}]\rho(E)dE=\int\mathbb{E}[G(E)1_{|E|\leq 2-\epsilon}]\rho(E)dE.
\end{equation}
The left hand side is equal to, via Fubini,
$$\mathbb{E}\sum_{\mu\in\Lambda_n}g_2(\mu,\boldsymbol{\psi}_n^\mu)\int_{-2+\epsilon}^{2-\epsilon}g_1(n\rho(E)(\mu-E))\rho(E)dE.$$ Now set some $\delta>\epsilon$ and consider $A_n(\delta)=\{\mu\in \Lambda_n:|\mu|<2-\delta\}$ and $B_n(\delta)=\{\mu\in\Lambda_n:|\mu|\geq 2-\delta\}$. The summation over $B_n(\delta)$ is bounded thanks to Lemma 7.4 of \cite{nakano2022shape}:
$$\lim_{n\to\infty}\frac{1}{n}\mathbb{E}|B_n(\delta)|=\frac{1}{\pi}\int_{2-\delta}^2 \rho(s)ds\leq C\sqrt{\delta},$$ and that for $n$ large and $\mu\in A_n(\delta)$,
$$\int_{-2+\epsilon}^{2-\epsilon}g_1(n\rho(x)(x-\mu))\rho(x)dx=\frac{1}{n}\int g_1(x)dx+o(1/n)=\frac{1}{n}+o(1/n),
$$ see \cite{nakano2022shape}, Section 4 for the proof of this claim. Finally, we have
$$\int\mathbb{E}[G(E)1_{|E|<2-\epsilon}]\rho(E)dE=\int \mathbb{E}[G(E)]\rho(E)dE+O(\epsilon).$$
Combining everything, the proof of the theorem follows.
\end{proof}

\section{Properties of the point process}\label{section6}
In this section we give the proof of Propositions \ref{prop1.5}, \ref{prop1.6}, \ref{prop1.7}, characterizing the properties of the $\eta Sch$ point process we just introduced. These properties are similar to those of the $Sch_\tau$ process discussed in \cite{kritchevski2011scaling}.

First we prove Proposition \ref{prop1.5}, repulsion of nearby eigenvalues.

\begin{proof}
In the proof of Lemma \ref{lemma3.6!} we have illustrated that $\theta$ has the following invariance property: for every $u\in\mathbb{R}$, 
$$\theta^{\lambda-u}(t)+ut=^d \theta^\lambda(t).$$
Thus for \eqref{largegap3.15} we only need to consider the case $\mu=0$.

Define the relative phase function $\alpha^\lambda(t)=\theta^\lambda(t)-\theta^0(t)$, then $\alpha^\lambda(t)$ solves the SDE
\begin{equation}\label{6.1good}
    d\alpha^\lambda(t)=\lambda dt+\frac{\sigma\rho}{(1-t)^{\frac{1}{2}-\eta}}\text{Im}[(e^{-i\alpha^\lambda(t)}-1)d\mathcal{Z}],\quad \alpha^\lambda(0)=0,
\end{equation}
where $\mathcal{Z}$ satisfies $d\mathcal{Z}=e^{-i\theta^0}d\mathcal{W}$ and $\mathcal{W}$ is another complex Brownian motion.
We shall rewrite it as
\begin{equation}\label{eq29}
    d\alpha^\lambda=\lambda dt+\sqrt{2}\frac{\sigma\rho}{(1-t)^{\frac{1}{2}-\eta}}\sin(\alpha^\lambda/2)d\mathcal{B}^\lambda,\quad \alpha^\lambda(0)=0,
\end{equation}
with $\mathcal{B}^\lambda$ satisfying $d\mathcal{B}^\lambda=-\sqrt{2}\text{Re}[e^{-i\alpha^\lambda/2}d\mathcal{Z}]$ is a standard Brownian motion.

The relative phase $\alpha^\lambda(t)$ provides a very nice way to estimate the number of points of $\eta Sch.$ This follows from the inequality
\begin{equation}
    \left|\frac{1}{2\pi}\alpha^\lambda(1)-\eta Sch[0,\lambda]\right|\leq 1.
\end{equation}

Now we fix $\epsilon>0$ and define a process $Y=\log(\tan(\alpha/4))$. From Itô's formula $Y$ solves the SDE
\begin{equation}\label{sdeforY}
    dY=\frac{\epsilon}{2}\cosh(Y)dt+\frac{1}{4}\frac{\sigma^2\rho^2}{(1-t)^{1-2\eta}}\tanh(Y)dt+\frac{1}{\sqrt{2}}\frac{\sigma\rho}{(1-t)^{\frac{1}{2}-\eta
    }}dB,
\end{equation}
with initial value $Y(0)=-\infty.$  By definition, $\{\alpha^\epsilon(1)\geq 2\pi\}=\{Y\text{ explodes on }[0,1]\}$. Now define the solution $\tilde{Y}$ that solves the same SDE \eqref{sdeforY}, but with $\tilde{Y}(0)=0.$ Then 
$$Y\text{ explodes on }[0,1]\Rightarrow \tilde{Y} \text{ explodes on }[0,1]\Rightarrow \sup_{t\in[0,1]}|\tilde{Y}(t)|\geq \log (1/\epsilon).$$
Given $\epsilon\leq 1$, $|y|\leq\log(1/\epsilon)$ leads to
$$|\epsilon\cosh(y)/2|\leq 1,\quad |\tanh(y)/4|\leq 1.$$

From this we get, for any $s\geq 0$,
$$\sup_{t\in[0,s]}|\tilde{Y}(t)|\leq \log(1/\epsilon)\Rightarrow \sup_{t\in[0,s]}|\tilde{Y}(t)-
\frac{1}{\sqrt{2}}Z(t)|\leq d(s),$$
where we define $Z(t)=\int_0^t \frac{\sigma\rho}{(1-s)^{\frac{1}{2}-\eta}}dB_s$ and $d(s)=s+\int_0^s \frac{\sigma^2\rho^2}{(1-t)^{1-2\eta}}dt,$ so that, denoting $T$ the first time $\tilde{Y}$ reaches $\log(1/\epsilon)$, we must have $\frac{1}{\sqrt{2}}|Z(T)|\geq\log(1/\epsilon)-d(T),$ which implies
$$\begin{aligned}
\mathbb{P}(\sup_{t\in[0,1]}|\tilde{Y}(t)|\geq\log(1/\epsilon))&\leq\mathbb{P}(\frac{1}{\sqrt{2}
}\sup_{t\in[0,1]}|Z(t)|\geq\log(1/\epsilon)-d(1))\\&\leq \exp(-C_\eta(\log(1/\epsilon)-d(1))^2).\end{aligned} 
$$ Since $\eta\in(0,\frac{1}{2}]$, we know that $d(1)<\infty.$
In the last inequality, we use that $\sup_{t\in[0,1]}|Z(t)|$ has sub-exponential tails. This follows from representing $Z(t)$ as a time-changed Brownian motion via Dambis-Dubins-Schwarz Theorem: $Z(t)\sim_{law}\mathcal{B}_{a(t)}$ for another Brownian motion $\mathcal{B}$ and $a(t)=\sigma\rho\sqrt{\int_0^t\frac{ds}{(1-s)^{1-2\eta}}}$, and then use the reflection principle of Brownian motion and Brownian scaling, together with the fact that $a(1)=O(1)$.
\end{proof}

Then we prove Proposition \ref{prop1.6}, the chance of a very long gap in the spectrum.
\begin{proof}
    Let $\alpha=\alpha^\lambda$ be defined in \eqref{6.1good}. The desired probability can be bounded via
    \begin{equation}
        \mathbb{P}(\theta^0(1)\in(0,\epsilon)\text{ mod }2\pi,\text{ and }\alpha(1)\leq 2\pi-\epsilon)\leq\mathbb{P}(\eta Sch[0,\lambda]=0)\leq\mathbb{P}(\alpha(1)\leq 2\pi).
    \end{equation}

Recall \eqref{thetalambda} that $$\theta^0(1)=\int_0^1 \frac{\sigma\rho}{(1-t)^{\frac{1}{2}-\eta}}d\mathcal{B}-\int_0^1 \frac{\sigma\rho}{(1-t)^{\frac{1}{2}-\eta}}\text{Im}[d\mathcal{Z}]$$ and $\mathcal{Z}$ satisfying $d\mathcal{Z}_t=e^{-i\theta^0(t)}d\mathcal{W}_t$ is thus independent of $\mathcal{B}$. Therefore by independence, 
$$\mathbb{E}[1\{\theta^0(1)\in(0,\epsilon)\text{ mod }2\pi\}\mid\mathcal{Z}]\geq\epsilon\min_x f_{2\pi}(x)=c\epsilon>0,$$ uniform in $\lambda$, where $f_{2\pi}(x)$ is the density of $\int_0^1 \frac{\sigma\rho}{(1-t)^{\frac{1}{2}-\epsilon}}d\mathcal{B}$ mod $2\pi$. Thus we can lower bound the left hand side by $c\epsilon \mathbb{P}(\alpha(1)\leq 2\pi-\epsilon),$ where $c$ depends implicitly on $\eta,\sigma,\rho$, 

The rest of the argument is about estimating the asymptotic of $\mathbb{P}(\alpha^\lambda(1)\leq 2\pi)$ as $\lambda\to\infty$. Recall that $\alpha^\lambda$ solves \begin{equation}
    d\alpha^\lambda=\lambda dt+\sqrt{2}\frac{\sigma\rho}{(1-t)^{\frac{1}{2}-\eta}}\sin(\alpha^\lambda/2)d\mathcal{B}^\lambda,\quad \alpha^\lambda(0)=0.
\end{equation} 

To make better use of existing results, we use the time change as in equation \eqref{timechange}. Then $\alpha^\lambda$ solves  
\begin{equation}
    d\alpha^\lambda=\frac{2(1-\frac{4\eta v}{\sigma^2\rho^2})^{\frac{1-2\eta}{2\eta}}}{\sigma^2\rho^2}\lambda dv+2\sin(\frac{\alpha^\lambda}{2})d\mathcal{B}^\lambda_v,\quad \alpha^\lambda(0)=0,
\end{equation} for $v\in[0,v(1)]$. 

By adapting the proof of \cite{valko2009continuum}, Theorem 13 to our setting, where they considered the SDE $$d\tilde{\alpha}=\lambda fdt+2\sin(\tilde{\alpha}/2)dB,\quad\tilde{\alpha}(0)=0$$
and obtained a two-sided estimate $\exp(-\lambda^2\|f\|_2^2/8+o(1)),$
we see that we may find some $c_\eta>0$ such that 
$$\mathbb{P}(\alpha^\lambda(1)\leq 2\pi)=\exp(-c_\eta\frac{\lambda^2}{{4\sigma^2\rho^2}}(1+o(1))).$$
To check the last expression, in our setting we need to compute 
$$\|f\|_2^2=
\int_0^\frac{\sigma^2\rho^2}{4\eta}\frac{4(1-\frac{4\eta v}{\sigma^2\rho^2})^\frac{1-2\eta}{\eta}}{\sigma^4\rho^4}dv=c_\eta\frac{1}{\sigma^2\rho^2}
$$ for some $c_\eta>0$. This finishes the proof. A careful check of the time change shows that the constant $c_\eta$ depends only on $\eta>0$ but not $\sigma$ and $\rho$. When we consider random Schrödinger operators with more general decay speed profile, the resulting constant (in place of $c_\eta$) still depends only on the decaying profile and not $\sigma,\rho$, but the expression $c_\eta$ could be rather complicated.  

\end{proof}

Finally we prove Proposition \ref{prop1.7}, on the fluctuation of number of points in the $\eta Sch$ point process.

\begin{proof}
Recall the SDE solved by $\alpha^\lambda$ in \eqref{6.1good}. Since $\int_0^1 \frac{\sigma^2\rho^2}{(1-t)^{1-2\eta}}dt<\infty,$ we deduce that $\alpha^\lambda(1)-\lambda$ is a Gaussian with finite variance. Then the claim follows from, with $\lambda=2\pi k+\theta$, 
$$ \eta Sch[0,2\pi k+\theta]-k=\#\{[\theta^0(1),\theta^\lambda(1)]\cap 2\pi\mathbb{Z}\}-k $$
and an application of Markov's inequality.
\end{proof}

When the potential decays much faster in the setting of Corollary \ref{reallyfast}, the proof is a straightforward modification of the previous proof:

\begin{proof}[\proofname\ of Corollary \ref{reallyfast}]The proof architecture in all the previous arguments follow without change, and one can check via computing the quadratic variation that the point process $\Lambda_{n,E}-\arg(z^{2n+2})-\pi$ converges in distribution to the $\sigma=0$ version of the point process $\eta Sch$, but this $\sigma=0$ version is previously given by $d\theta^\lambda(t)=\lambda dt$, so that $\theta^\lambda(t)=\lambda t$ and thus $\eta Sch^\phi=\{\lambda:\lambda\in 2\pi\mathbb{Z}\}$. Thus $\Lambda{n,E}-z^{2n+2}$ converges in distribution to $2\pi\mathbb{Z}$. The other arguments are exactly the $\sigma=0$ version of previous results and thus omitted.
    
\end{proof}

\section*{Conflict of interest statement}
The author does not declare any conflict of interest with other group of researchers, related to the content and completion procedure of this manuscript.

\section*{Data Availability Statement}
No dataset is generated during the completion of this manuscript. The cited references contain the main supporting materials for the scientific ground of this manuscript.

\bibliographystyle{amsplain}

\end{document}